\documentclass[12pt,a4paper]{amsart}
\usepackage[a4paper,margin=1in]{geometry}
\usepackage{amsmath, amssymb}
\usepackage{graphicx}
\usepackage{tikz}
\usepackage{pgfplots}
\pgfplotsset{compat=1.18}
\usepackage{hyperref}
\numberwithin{equation}{section}
\newtheorem{theorem}{Theorem}[section]
\newtheorem{corollary}{Corollary}[section]
\newtheorem{definition}{Definition}[section]
\newtheorem{lemma}{Lemma}[section]
\newtheorem{example}{Example}[section]
\newtheorem{proposition}{Proposition}[section]
\theoremstyle{remark}
\newtheorem{remark}{Remark}[section]

\newcommand{\D}{\mathbb{D}}

\title[Volterra-type integral operator]
{Volterra-type integral operator on analytic function spaces}


\newcounter{minutes}
\divide\time by 60
\newcounter{hours}
\multiply\time by 60
\addtocounter{minutes}{-\time}

\begin{document}
\begin{abstract}
We investigate the geometric properties of the Volterra-type integral operator
\begin{equation*}
T_g[f](z) = \int_{0}^{z} f(s)\, g'(s)\, ds, \quad |z|<1,
\end{equation*}
acting on various subclasses of analytic functions in the unit disk.  
Sharp estimates are obtained for the convexity radius of $T_g$, which simultaneously determine its univalence radius, across several classical function families.  
In addition, we introduce and study higher-order Volterra-type operators, establish their normalized forms, and propose an open question on the scaling behavior of their convexity radii.  
\end{abstract}

\def\thefootnote{}
\phantomsection
\footnotetext{
\texttt{\tiny File:~\jobname .tex,
          printed: \number\year-\number\month-\number\day,          \thehours.\ifnum\theminutes<10{0}\fi\theminutes}}
\makeatletter\def\thefootnote{\@arabic\c@footnote}\makeatother

\author[R. Kargar] {Rahim Kargar}
\address{Department of Mathematics and Statistics, University of Turku,
         Turku, Finland}
\email{rahim.r.kargar@utu.fi}

\maketitle

\section{Introduction}

The study of integral operators acting on spaces of analytic functions has been a central theme in complex analysis and operator theory. A particularly important example is the Volterra-type operator, introduced by Pommerenke in 1977~\cite{Pommerenke1977}, defined for analytic functions \(f\) and \(g\) in the open unit disk \(\mathbb{D}\) by
\begin{equation}\label{eq:Tg}
T_g[f](z):= \int_0^z f(s) g'(s)\, ds, \quad |z|<1.
\end{equation}
Pommerenke showed that \(T_g\) is bounded on the Hardy space \(H^2\) if and only if \(g\) belongs to \(\mathrm{BMOA}\). This characterization was later extended by Aleman and Siskakis~\cite{Aleman} to all Hardy spaces \(H^p\) (\(1 \le p < \infty\)), and the corresponding compactness criterion was established using \(\mathrm{VMOA}\).
A closely related operator to the Volterra-type integral operator $T_g$ is
\begin{equation*}
J_g[f](z) := \int_0^z f'(s) g(s)\, ds, \quad z\in\mathbb{D},
\end{equation*}
which satisfies the identity
\[
T_g[f](z) + J_g[f](z) = M_g [f](z) - f(0) g(0),
\]
where \(M_g\) denotes the multiplication operator \(M_g [f](z)=g(z) f(z)\). %In particular, if \(f(0)=g(0)=0\) and \(f'(0)=g'(0)=1\), then
%\[
%T_g[f](z) + J_g[f](z) = g(z) f(z).
%\]
These operators naturally act on various spaces of analytic functions, including Hardy, Bergman, Bloch, and Zygmund spaces, and they encompass several well-known operators in geometric function theory, such as the generalized Bernardi--Libera--Livingston operators~\cite{Bernardi,Libera,Livingston}, the Srivastava--Owa fractional derivative operators~\cite{Owa1978,OwaSri}, and the Ces\`{a}ro operator~\cite{Siskakis1987,Siskakis1990}.  

From the perspective of geometric function theory, a natural question is how the analytic and geometric properties of \(f\) and \(g\) influence the univalence of \(T_g\). Notably, even if both \(f\) and \(g\) are univalent in \(\mathbb{D}\), their Volterra composition \(T_g\) need not be univalent. A simple example is \(f(z)=g(z)=z\), for which \(T_g[f](z)=z^2/2\) is clearly not univalent. This observation motivates the study of the radius of convexity (or injectivity) for \(T_g\), measuring how far from the origin the operator preserves convexity or univalence.

In this paper, of course, we focus on sharp estimates for the radius of convexity of the Volterra-type integral operator $T_g$ when \(f\) belongs to the class
\(\mathcal{P}\) of analytic functions with positive real part and $f(0)=1$ while \(g\) belongs to several well-known
subclasses of analytic functions.
Furthermore, we introduce and study the \textit{normalized} higher-order Volterra-type operator
\begin{equation*}
\widetilde T_{g,n}[f](z)
:=\frac{n!}{g'(0)}\,
\frac{T_{g,n}[f](z)}{z^{\,n-1}},\quad n=1,2,\ldots,\, g'(0)\neq0,
\end{equation*}
where 
\begin{equation*}
T_{g,n}[f](z)
= \underbrace{\int_0^z \int_0^{t_1} \!\!\cdots \int_0^{t_{n-1}}}_{n\ \text{times}}
f(t_n)\, g'(t_n)\, dt_n \cdots dt_1,
\quad z\in\mathbb{D},
\end{equation*}
such that \(\widetilde{T}_{g,n}[f](0) =0\) and \((\widetilde{T}_{g,n}[f])'(0)=1.\)
We then investigate its geometric properties and formulate an open question
describing how the radius of convexity scales with the order \(n\).

The structure of the paper is as follows. 
In Section~\ref{sec:Subclass}, we review important subclasses of analytic functions and collect key lemmas. 
Section~\ref{sec:convexity radius-T_f} establishes sharp results for the radius of convexity of \(T_g\) for various combinations of \(f\) and \(g\). 
Finally, Section~\ref{sec:higher-order} extends these results to higher-order Volterra-type operators, introduces their normalized versions, and concludes with an open question on the asymptotic scaling behavior of their convexity radii.

\section{Preliminaries}\label{sec:Subclass}

In this section, we recall several important subclasses of analytic functions that will be used throughout the paper. We use $``:="$ to denote the definition of the left-hand side by the right-hand side.

Let 
\[
\mathbb{D} := \{ z \in \mathbb{C} : |z|<1 \}
\] 
denote the open unit disk and let $\mathcal{H}(\mathbb{D})$ be the class of all analytic functions in $\mathbb{D}$. The class of all normalized analytic functions in $\mathbb{D}$ is denoted by $\mathcal{A}$, that is,
\[
f \in \mathcal{A} \quad \Longleftrightarrow \quad f \in \mathcal{H}(\mathbb{D}), \; f(0)=0, \; f'(0)=1.
\]
Let $\mathcal{U} \subset \mathcal{A}$ denote the subclass of univalent (one-to-one) functions in $\mathbb{D}$.  
For $f,g \in \mathcal{H}(\mathbb{D})$, we say that $f$ is subordinate to $g$, written $f(z)\prec g(z)$, if there exists a Schwarz function $w$ with $w(0)=0$ and $|w(z)|\le 1$ such that
\[
f(z) = g(w(z)), \quad z \in \mathbb{D}
\]
for all $z\in\mathbb{D}$. 
In particular, if $g$ is univalent, then
\[
f(0) = g(0) \quad \text{and} \quad f(\mathbb{D}) \subset g(\mathbb{D}).
\]

\subsection*{Carath\'eodory and convex functions}
Let $\mathcal{P}$ denote the Carath\'eodory class of functions $p \in \mathcal{H}(\mathbb{D})$ satisfying
\[
p(0)=1 \quad \text{and} \quad {\rm Re}\{p(z)\}>0, \quad z \in \mathbb{D}.
\]
A function $f \in \mathcal{A}$ is convex of order $\alpha$, $0 \le \alpha < 1$, if
\[
{\rm Re} \left\{ 1 + z \frac{f''(z)}{f'(z)} \right\} > \alpha, \quad z \in \mathbb{D}.
\]
The class of convex functions of order $\alpha$ is denoted by $\mathcal{K}(\alpha)$.

\subsection*{Janowski-type starlike and convex functions}
Let $A,B \in \mathbb{C}$ with $A\neq B$ and $|B|\le 1$. A function $f \in \mathcal{A}$ belongs to the class $\mathcal{S}^*_G(A,B)$ if
\[
z \frac{f'(z)}{f(z)} \prec \frac{1+Az}{1+Bz}, \quad z \in \mathbb{D}.
\]
This generalizes Janowski starlike functions; if $A$ and $B$ are real with $-1 \le B < A \le 1$, we recover the classical Janowski starlike class. For complex parameters, univalence is not guaranteed (e.g., $A=2$, $B=1$ or $B=-1$).  

The associated convex class is
\[
f \in \mathcal{K}_G(A,B) \quad \Longleftrightarrow \quad z f'(z) \in \mathcal{S}^*_G(A,B).
\]
In particular, $\mathcal{K}_G(2,1)$ and $\mathcal{K}_G(2,-1)$ reduce to classical Ozaki conditions. By Lindel\"of's subordination principle, if $f \in \mathcal{K}_G(2,1)$, then
\begin{equation}\label{Ozaki 3 2}
{\rm Re}\left\{ 1 + z \frac{f''(z)}{f'(z)} \right\} < \frac{3}{2}, \quad z \in \mathbb{D}
\end{equation}
and if $f \in \mathcal{K}_G(2,-1)$, then
\begin{equation}\label{Ozaki 1 2}
{\rm Re}\left\{ 1 + z \frac{f''(z)}{f'(z)} \right\} > -\frac{1}{2}, \quad z \in \mathbb{D}.
\end{equation}
Ozaki \cite{Ozaki} showed that if $f \in \mathcal{A}$ with $f(z)f'(z)/z \neq 0$ and either \eqref{Ozaki 3 2} or \eqref{Ozaki 1 2} holds, then $f$ is univalent and convex in at least one direction.  

Let
\[
\mathcal{P}[A,B] := \left\{ p \in \mathcal{P} : p(z) \prec \frac{1+Az}{1+Bz} \right\}.
\]

\subsection*{Locally univalent and linear invariant families}
Let $\mathcal{LU}$ denote the class of normalized locally univalent functions. For $\beta \in \mathbb{R}$, define
\[
\mathcal{G}(\beta) := \left\{ f \in \mathcal{LU} : {\rm Re} \left\{ 1 + z \frac{f''(z)}{f'(z)} \right\} < 1 + \frac{\beta}{2}, \; z \in \mathbb{D} \right\}.
\]
It is known that $\mathcal{G}(1) \subset \mathcal{U}$, and the functions in $\mathcal{G}(1)$ are starlike \cite{KargarJMAA,OPW}.

Let $Aut(\mathbb{D})$ denote the group of holomorphic automorphisms of $\mathbb{D}$. A family $\mathcal{F} \subset \mathcal{A}$ is a linear invariant family (LIF) if $\mathcal{F} \subset \mathcal{LU}$ and
\[
F_\phi(f)(z) := \frac{f(\phi(z)) - f(\phi(0))}{f'(\phi(0)) \phi'(0)} \in \mathcal{F}, \quad f \in \mathcal{F}, \; \phi \in Aut(\mathbb{D}).
\]
The order of $\mathcal{F}$ is defined by
\[
{\rm ord}\,\mathcal{F} := \sup_{f \in \mathcal{F}} \left| \frac{f''(0)}{2} \right|.
\]
%and the universal linear invariant family of order $\gamma \ge 1$ is
%\[
%\mathcal{UL}_\gamma := \{ f \in \mathcal{F} : {\rm ord}(\mathcal{F}) \le \gamma \}.
%\]
%In particular, $\mathcal{UL}_1 \equiv \mathcal{K}(0)$ and $\mathcal{U} \subset \mathcal{UL}_2$ \cite[Ch.~5]{Graham}.

\subsection*{Functions of bounded boundary rotation}
For $k \ge 2$, let $V_k$ denote the class of functions $f \in \mathcal{A}$ such that
\[
\int_0^{2\pi} \left| {\rm Re} \left\{ 1 + z \frac{f''(z)}{f'(z)} \right\} \right|_{z=re^{i\theta}} d\theta \le k \pi.
\]
Functions in $V_k$ map $\mathbb{D}$ conformally onto domains with boundary rotation at most $k \pi$. This class was introduced by Loewner \cite{Low} and developed by Paatero \cite{Paa1}. Robertson \cite[Theorem 1]{Robertson1969} proved that for $f \in V_k$,
\begin{equation}\label{ineq Vk}
\left| z \frac{f''(z)}{f'(z)} - \frac{2|z|^2}{1-|z|^2} \right| \le \frac{k |z|}{1-|z|^2}, \quad z \in \mathbb{D}, \; 2 \le k < \infty.
\end{equation}

To validate our findings, we require the following lemmas. %The first lemma gives a basic estimate that leads to the distortion theorem for univalent functions.

%\begin{lemma}\label{lem estimate starlike functions}
%  {\rm(see \cite[p. 31]{Graham})} 
%  If $f\in \mathcal{U}$, then for $|z|=r<1$,
 % \begin{equation*}
 %   \left| z\frac{f'(z)}{f(z)} - \frac{1+r^2}{1-r^2} \right| \le \frac{2r}{1-r^2}.
%  \end{equation*}
%  The inequality is sharp for  the Koebe function or one of its rotations.
%\end{lemma}

\begin{lemma}\label{lem estimate univalent functions}{\rm(}see \cite[p. 15]{Graham}{\rm)}
  If $f\in \mathcal{U}$, then
\begin{equation*}
  \left|z\frac{f''(z)}{f'(z)}-\frac{2r^2}{1-r^2}\right|\leq \frac{4r}{1-r^2},\quad|z|=r<1.
\end{equation*}
The estimate is sharp for the rotation of the Koebe function $k(z)=z/(1-z)^2$.
\end{lemma}

%**********************************************
\begin{lemma}\label{lem Obradovic2013}{\rm(}see \cite{OPW}{\rm)}
Let $\beta\in(0,1]$ be fixed. If $f\in \mathcal{G}(\beta)$, then
\begin{equation*}
  \left|\frac{f''(z)}{f'(z)}\right|\leq \frac{\beta}{1-|z|},\quad z\in\mathbb{D}.
\end{equation*}
The result is sharp for the function $f'(z)=(1-z)^\beta$.
\end{lemma}
%**********************************************
The next lemma is due to Pommerenke \cite{Pom}, see also \cite[Lemma 5.1.3]{Graham}.
%---------------------------------------------------------------------
\begin{lemma}\label{lem linear invariant}
  Let $\mathcal{F}$ be a linear invariant family and $\delta={\rm ord}\mathcal{F}$. Then
\begin{equation*}\label{delta= sup sup}\delta=\sup_{f\in\mathcal{F}}\sup_{|z|<1}\left|-\overline{z}
    +\frac{1}{2}(1-|z|^2)\frac{f''(z)}{f'(z)}\right|.
  \end{equation*}
\end{lemma}
%---------------------------------------------------------------------
%\begin{lemma}\label{lem Ebadian, Kargar}{\rm (}see \cite{ebakar2017}{\rm )}
%If $f\in\mathcal{UL}_{\gamma}$ and $\gamma\geq1$, then
%\begin{equation*}
%\left|z\frac{f''(z)}{f'(z)}-\frac{2|z|^{2}}{1-|z|^{2}}\right|
%\leq\frac{2\gamma|z|}{1-|z|^{2}},\quad z\in\mathbb{D}.
%\end{equation*}
%\end{lemma}

We next present a classical sharp estimate, along with its proof, for functions in the Carath\'eodory--Janowski class, which will play a crucial role in proving the main results.

\begin{lemma}\label{lem:est-Cara-Janowski-sharp}
Let $p\in \mathcal{P}[A,B]$ with real parameters $-1\le B<A\le 1$. For $|z|=r<1$, we have the sharp lower and upper bounds
\begin{equation}\label{eq:PAB-bound-sharp}
-\frac{(A-B)r}{(1-A r)(1-B r)}
\le\ {\rm Re}\left\{ z\frac{p'(z)}{p(z)} \right\} \le\frac{(A-B)r}{(1-A r)(1-B r)}.
\end{equation}
Equality is achieved on the positive real axis by extremal functions
\[
p_\pm(z)=\frac{1 \pm A z}{1 \pm B z}, \quad \text{or by rotations } p(z)=\frac{1+\varepsilon A z}{1+\varepsilon B z},\ |\varepsilon|=1.
\]
\end{lemma}

\begin{proof}
Since $p\in \mathcal{P}[A,B]$, by the subordination principle, there exists a Schwarz function $w:\mathbb{D}\to\overline{\mathbb{D}}$ with $w(0)=0$ such that
\[
p(z)=\frac{1+A w(z)}{1+B w(z)}, \quad z\in \mathbb{D}.
\]
A direct computation gives
\[
z\frac{p'(z)}{p(z)} = \frac{(A-B) z w'(z)}{(1+A w(z))(1+B w(z))}.
\]
By the Schwarz lemma, \(|w(z)| \le |z|\) and by the Schwarz–Pick lemma, we have
\begin{equation*}
|w'(z)| \le \frac{1-|w(z)|^2}{1-|z|^2},\quad  z\in\mathbb D.
\end{equation*}
Using these, we have
\begin{align*}
\left| z \frac{p'(z)}{p(z)} \right|
\le \frac{(A-B) |z|}{(1-A |z|)(1-B |z|)}
\le \frac{(A-B) r}{(1-A r)(1-B r)}.
\end{align*}
Therefore, the real part satisfies
\[
-\frac{(A-B)r}{(1-A r)(1-B r)} \le {\rm Re}\left\{ z \frac{p'(z)}{p(z)} \right\} \le \frac{(A-B)r}{(1-A r)(1-B r)}.
\]
Sharpness is obtained for the extremal functions
\[
p_\pm(z) = \frac{1 \pm A z}{1 \pm B z},
\]
or by one of its rotations for which the real part achieves the bounds on the positive real axis. This completes the proof. 
\end{proof}

If we take $A=1$ and $B=-1$ in Lemma \ref{lem:est-Cara-Janowski-sharp}, we get:
\begin{corollary}%\label{lem:est-Cara}
Let \(p\) be a Carath\'eodory function, that is, \(p\in\mathcal P\).  
For \(|z|=r<1\) we have the sharp estimate
\[
\left| z\frac{p'(z)}{p(z)}\right| \le \frac{2r}{1-r^2},
\]
and consequently
\[
-\frac{2r}{1-r^2} \le {\rm Re}\left\{ z\frac{p'(z)}{p(z)} \right\} \le \frac{2r}{1-r^2}.
\]
Equality is attained (for each fixed \(r\)) by the extremal functions
\[
p_\varepsilon(z)=\frac{1+\varepsilon z}{1-\varepsilon z},\quad |\varepsilon|=1,
\]
with the choice \(\varepsilon=1\) giving the upper bound and \(\varepsilon=-1\) giving the lower bound on the positive real axis.
\end{corollary}

%*********************************************************
\section{The Convexity Radius of \texorpdfstring{$T_g$}{Tg}}\label{sec:convexity radius-T_f}
In this section, we obtain the sharp convexity radius of the Volterra-type integral operator $T_g$ when $g$ belongs to some well-known subclasses of analytic functions. 
Moreover, the first result of this section is as follows.
\begin{theorem}\label{thm:rad-convex-AB}
Let $A,B\in\mathbb{R}$ satisfy $-1\le B<A\le 1$ and let $0\le\alpha<1$. 
If $f\in\mathcal{P}[A,B]$ and $g\in\mathcal{K}_G(A,B)$, then the Volterra--type integral operator $T_g$
is convex of order $\alpha$ in the disk $|z|<r_c(A,B,\alpha)$, where \(r_c(A,B,\alpha)\) is the unique root of the equation
\begin{equation}\label{eq:root-equation}
\frac{1-(A-B)r-AB r^2}{1-B^2 r^2}-\frac{(A-B)r}{(1-A r)(1-B r)} - \alpha=0.
\end{equation}
Moreover, the result is sharp.
\end{theorem}

\begin{proof}
A direct computation using \eqref{eq:Tg} yields the identity
\begin{equation}\label{eq:1+zT''}
1+z\frac{T_g''(z)}{T_g'(z)}
= z\frac{f'(z)}{f(z)} + 1+z\frac{g''(z)}{g'(z)}.
\end{equation}
It is clear that if $f\in\mathcal{P}[A,B]$ and $g\in\mathcal{K}_G(A,B)$, then $T_g\in\mathcal{A}$, thus, $T_g(0)=T'_g(0)-1=0$. Therefore, 
we will find the maximal $r\in(0,1)$ such that
\begin{equation*}
\inf_{|z|=r}{\rm Re}\,\left\{1+z\frac{T_g''(z)}{T_g'(z)}\right\}=\inf_{|z|=r}{\rm Re}\,\left\{ z\frac{f'(z)}{f(z)} + 1+z\frac{g''(z)}{g'(z)}\right\}>\alpha.
\end{equation*}
Since \(f\in\mathcal P[A,B]\) and \(A,B\) are real with \(-1\le B<A\le1\), Lemma \ref{lem:est-Cara-Janowski-sharp} gives the sharp one–sided bound
\[
{\rm Re}\,\left\{ z\frac{f'(z)}{f(z)}\right\} \ge -\frac{(A-B)r}{(1-A r)(1-B r)}, \quad |z|=r.
\]
By hypothesis \(g\in\mathcal K_G(A,B)\); hence
\begin{equation*}
q(z):=1+z\frac{g''(z)}{g'(z)}\in\mathcal{P}[A,B].
\end{equation*}
Applying the Janowski pointwise disk to \(q\) (as in the standard Janowski estimate), we obtain the sharp lower bound
\[
{\rm Re}\,\{q(z)\} \ge \frac{1-(A-B)r-AB r^2}{1-B^2 r^2},\quad |z|=r.
\]
Adding the two lower bounds and using the identity \eqref{eq:1+zT''} yields, for \(|z|=r\),
\[
{\rm Re}\,\left\{1+z\frac{T_g''(z)}{T_g'(z)}\right\}
\ge -\frac{(A-B)r}{(1-A r)(1-B r)} + \frac{1-(A-B)r-AB r^2}{1-B^2 r^2}=:H(r).
\]
However, $T_g$ is convex of order alpha if and only if $H(r,\alpha):=H(r)-\alpha>0$.
Note that \(H(0,\alpha)=1-\alpha>0\) and, \(H(r,\alpha)\to -\infty\) as \(r\to1^{-}\). 
Moreover, \(H(r,\alpha)\) is continuous and strictly decreasing on \((0,1)\), hence, there exists a unique \(r_c:=r_c(A,B,\alpha)\in(0,1)\) satisfying \(H(r_c,\alpha)=0\). Therefore, $T_g$ is convex of order $\alpha$ in the disk $|z|<r_c(A,B,\alpha)$.

To show sharpness, take the Janowski extremals
\[
f_0(z)=\frac{1-Az}{1-Bz}\in\mathcal P[A,B],\quad g_0(z)=\frac{1}{A}\left(1-(1-Bz)^{\frac{A}{B}}\right),\quad A\neq0\neq B,
\]
where $g_0\in \mathcal{K}_G(A,B)$.
For these choices, we compute
\[
z\frac{f_0'(z)}{f_0(z)} = -\frac{(A-B)z}{(1-A z)(1-B z)},
\quad
1+z\frac{g_0''(z)}{g_0'(z)}=\frac{1-(A-B)z-AB z^2}{1-B^2 z^2},
\]
and consequently, by \eqref{eq:1+zT''},
\[
\operatorname{Re}\left\{1+z\frac{T_{g_0}''(z)}{T_{g_0}'(z)}\right\}
= H(|z|).
\]
In particular, for \(z=r>0\) real,
\[
\operatorname{Re}\left\{1+r\frac{T_{g_0}''(r)}{T_{g_0}'(r)}\right\}=H(r).
\]
If \(r_c\) is the unique root of \(H(r,\alpha)=0\), then the above shows
\[
\min_{|z|=r_c}\operatorname{Re}\left\{1+z\frac{T_{g_0}''(z)}{T_{g_0}'(z)}\right\}
=\alpha.
\]
Therefore, for every \(r>r_c\) we get \(H(r,\alpha)<0\), and the equality computation with the extremal shows there exists a point
on \(|z|=r\) (indeed the positive real point \(z=r\)) where
\(\operatorname{Re}\{1+zT_{g_0}''(z)/T_{g_0}'(z)\}<\alpha\). Consequently the radius \(r_c(A,B,\alpha)\) is best possible:
it cannot be replaced by any larger radius common to all admissible \((f,g)\). This completes the proof.
\end{proof}

\begin{remark}
    Here, we give two special cases of the above convexity radius $r_c(A,B,\alpha)$.\\
    \paragraph{\textbf{Case 1:} \(A=1,\; B=-1\).}
 Hence, \eqref{eq:root-equation} becomes
\[
\frac{1-2r+r^2}{1-r^2}-\frac{2r}{1-r^2}  - \alpha =0,
\]
or equivalently
\[
(1+\alpha) r^2 - 4r + 1-\alpha=0.
\]
The two roots of this quadratic are
\[
r_\pm=\frac{4\pm\sqrt{16-4(1+\alpha)(1-\alpha)}}{2(1+\alpha)}
=\frac{2\pm\sqrt{3+\alpha^2}}{1+\alpha}.
\]
For \(0\le\alpha<1\) the smaller root
\[
r_-=\dfrac{2-\sqrt{3+\alpha^2}}{1+\alpha}
\]
is the unique root in \((0,1)\) (the other root exceeds \(1\)). In particular, for \(\alpha=0\) this gives
\[
r=2-\sqrt3.
\]
Indeed, in this case the radius of convexity the Volterra--type integral operator $T_g$ is equal to familiar radius of convexity for the class $\mathcal{U}$ (see \cite[Theorem 2.2.22]{Graham}).\\

\paragraph{\textbf{Case 2:} \(B=0\) (general \(A\)).} The equation \eqref{eq:root-equation} becomes
\[
1-Ar-\frac{Ar}{1-Ar}-\alpha = 0
\]
or
\[
A^2 r^2 + A(\alpha-3)r + (1-\alpha) = 0.
\]
Hence, the roots are
\[
r_\pm=\frac{3-\alpha\pm\sqrt{\alpha^2-2\alpha+5}}{2A}.
\]
The admissible root (the one in \((0,1)\)) is the one that gives \(0<r<1\) for the given \(A>0\).  
A notable subcase: if \(A=1,\ \alpha=0\) the positive root is
\[
r=\frac{3-\sqrt5}{2},
\]
which is the familiar value that often appears in related radius problems.
\end{remark}

%****************************
%***********************************************

%********************************************
\begin{theorem}\label{theorem P[A,B] F}
Let $A,B\in\mathbb{R}$ with $-1\le B<A\le 1$.  
If $f\in\mathcal{P}[A,B]$ and $g\in \mathcal{F}$ with ${\rm ord}\mathcal{F}=\delta$,
then the Volterra-type integral operator $T_g$ is convex of order $\alpha$ in the disk $|z|<r_c(A,B,\alpha,\delta)$, where \(r_c(A,B,\alpha,\delta)\) is the unique root of the equation
\begin{equation*}\label{r c alpha gamma corrected final 2}
\frac{1-2\delta r+r^2}{1-r^2}-\frac{(A-B)r}{(1-A r)(1-B r)}-\alpha =0,
\end{equation*}
in the interval $(0,1)$.
The radius $r_c(A,B,\alpha,\delta)$ is the best possible.
\end{theorem}

\begin{proof}
If $g\in \mathcal{F}$, then by Lemma \ref{lem linear invariant}, we have
\begin{equation*}
 {\rm Re}\,\left\{ 1 + z \frac{g''(z)}{g'(z)} \right\}\geq \frac{1-2\delta r+r^2}{1-r^2},\quad(|z|=r).   
\end{equation*}
Therefore, since $f\in\mathcal{P}[A,B]$, by \eqref{eq:1+zT''} and by the last inequality above, we get 
\begin{equation*}
{\rm Re}\,\left\{1 + z \frac{T_g''(z)}{T_g'(z)} \right\}\geq -\frac{(A-B)r}{(1-A r)(1-B r)}+\frac{1-2\delta r+r^2}{1-r^2}=:K(r).
\end{equation*}
Consequently, \(T_g\) is convex of order \(\alpha\) in \(|z|<r\) whenever $K(r,\alpha):=K(r)-\alpha>0$.
%Thus, the admissible radii are precisely those \(r\) with \(K(r,\alpha)>0\), and the borderline radius is the (unique) solution of \(K(r,\alpha)=0\).
Note \(K(0,\alpha)=1-\alpha>0\) and tends to \(-\infty\) as \(r\to1^-\). Additionally, \(K\) is continuous and decreases on the interval \((0, 1)\). Therefore, there exists exactly one root, denoted by \(r_c(A, B, \alpha, \delta)\), in the interval \((0, 1)\). This indicates that \(r_c(A, B, \alpha, \delta)\) is the convexity radius of \(T_g\).

To show sharpness, we choose 
\[
f_0(z)=\frac{1-A z}{1-B z}\in\mathcal P[A,B],
\]
and 
\[
g_0(z)=\frac{1}{2\delta}\left[\left(\frac{1-z}{1+z}\right)^\delta-1\right],\quad {\rm ord}\mathcal F=\delta,
\]
which belongs to \(\mathcal F\). For these choices, one checks for real \(0<r<1\) that equality holds in both estimates used above:
\[
{\rm Re}\,\left\{ z\frac{f_0'(z)}{f_0(z)}\right\}= -\frac{(A-B)r}{(1-Ar)(1-Br)}, \quad 
{\rm Re}\,\left\{1+z\frac{g_0''(z)}{g_0'(z)}\right\}=\frac{1-2\delta r + r^2}{1-r^2}.
\]
Hence, for the pair \((f_0,g_0)\), the function \(K(r,\alpha)\) equals the left–hand side of the convexity condition and vanishes exactly at the unique root \(r_c(A,B,\alpha,\delta)\). Therefore, the radius cannot be improved in general, which proves sharpness.
\end{proof}
If we take $A=1$, $B=-1$, $\alpha=0$, and $\delta=1$ in the above Theorem \ref{theorem P[A,B] F}, we get:
%***********************************************
\begin{corollary}  
Let $f$ be a Carath\'eordory function and let $g$ be a convex univalent function. Then the Volterra-type integral operator $T_g$ is a convex univalent function in the disk $|z|<2-\sqrt{3}$, where \(2-\sqrt{3}\) is the unique root of the equation
\begin{equation*}
r^2 - 4r+1=0.
\end{equation*}
The result is sharp.
\end{corollary}

%********************************************
%********************************************
The next theorem is the following.
\begin{theorem}\label{theorem f P[A,B] g in U sharp}
Let $A,B\in\mathbb{R}$ with $-1\le B<A\le 1$.  
Let $f\in\mathcal{P}[A,B]$ and let $g$ be a univalent function. Then the Volterra-type integral operator $T_g$ is convex of order $\alpha$ in the disk $|z|<r_c(A,B,\alpha)$, where \(r_c(A,B,\alpha)\) is the unique root of the equation
\begin{equation}\label{eq:rc-AB-alpha-improved}
\frac{r^2 - 4r+1}{1-r^2}-\frac{(A-B)r}{(1-A r)(1-B r)}  - \alpha = 0.
\end{equation}
The radius $r_c(A,B,\alpha)$ is the best possible.
\end{theorem}

\begin{proof}
  Let $f\in\mathcal{P}[A,B]$ and $g \in \mathcal{U}$. From \eqref{eq:1+zT''}, Lemma \ref{lem:est-Cara-Janowski-sharp}, and the standard bound for univalent functions (Lemma \ref{lem estimate univalent functions}), we have
  \begin{align*}
    {\rm Re}\,\left\{ 1 + z \frac{T_g''(z)}{T_g'(z)} \right\} 
    &= {\rm Re}\,\left\{ \frac{z f'(z)}{f(z)} + 1 + z \frac{g''(z)}{g'(z)} \right\} \\
    &\geq -\frac{(A-B)r}{(1-A r)(1-B r)}+\frac{r^2-4r+1}{1-r^2}=:L(r).
  \end{align*}
  Hence, $T_g$ is convex of order alpha in the disk $|z| \le r<1$ if $L(r,\alpha):=L(r)-\alpha>0$.

We see that \(L(r,\alpha)\) is continuous on \([0,1)\), \(L(0,\alpha)=1-\alpha>0\), and \(\lim_{r\to1^-}L(r,\alpha)=-\infty\). A routine calculation shows \(L'(r,\alpha)<0\) for \(0<r<1\) under the parameter hypotheses \(-1\le B<A\le1\). Therefore, \(L(r,\alpha)\) is strictly decreasing, so there is exactly one root, denoted by \(r_c(A,B,\alpha)\), in \((0,1)\) such that \(L(r_c(A,B,\alpha),\alpha)=0\).

We show that the radius $r_c(A,B,\alpha)$ is the best possible. We take the Janowski extremal
\[
f_0(z)=\frac{1 - A z}{1 - B z}\in\mathcal P[A,B],
\]
and one of the rotations of the Koebe function. One convenient choice is
\[
g_0(z) = -k(-z) = \frac{z}{(1+z)^2},
\]
which is univalent and satisfies $g_0(0)=0$ and $g_0'(0)=1$.
 A direct computation gives for \(z=r>0\)
\[
{\rm Re}\,\left\{1 + r\frac{g_0''(r)}{g_0'(r)}\right\} = \frac{r^2-4r+1}{1-r^2},
\]
so equality is attained in the univalent bound at the positive real point \(r\).
Thus, for this pair \((f_0,g_0)\) the left-hand side of \eqref{eq:rc-AB-alpha-improved} equals zero exactly at \(r_c(A,B,\alpha)\). Therefore, the radius cannot be improved in general, and the result is sharp.

\end{proof}

%********************************************
\begin{remark}
  Taking $A=1$, $B=-1$, and $\alpha=0$ in Theorem \ref{theorem f P[A,B] g in U sharp}, we see that if $f$ is a Carath\'eodory function and $g$ is univalent, then the radius of convexity of the Volterra-type integral operator $T_g$ is $3-2\sqrt{2}$. 
\end{remark}

\begin{theorem}\label{theorem f starlike g in G beta sharp improved}
Let $A,B\in\mathbb{R}$ with $-1\le B<A\le 1$, and let $0\le \alpha<1$ and $\beta\in \mathbb R$.  
If $f\in\mathcal{P}[A,B]$ and $g\in\mathcal{G}(\beta)$, then the Volterra-type integral operator $T_g$ is convex of order $\alpha$ in the disk $|z|<r_c(A,B,\alpha,\beta)$,  
where \(r_c(A,B,\alpha,\beta)\) is the unique root of the equation
\begin{equation}\label{eq:rABalphaBeta}
1-\alpha-\frac{(A-B)r}{(1-A r)(1-B r)}-\frac{\beta r}{1-r}=0.
\end{equation}
The result is sharp.
\end{theorem}

\begin{proof}
  Let $f\in\mathcal{P}[A,B]$ and $g \in \mathcal{G}(\beta)$. It follows from \eqref{eq:1+zT''}, Lemma \ref{lem:est-Cara-Janowski-sharp}, and Lemma \ref{lem Obradovic2013} that
\[
{\rm Re}\!\left\{1+z\frac{T_g''(z)}{T_g'(z)}\right\}
\ge 1-\frac{(A-B)r}{(1-A r)(1-B r)}-\frac{\beta r}{1-r}.
\]
Therefore, \(T_g\) is convex of order \(\alpha\), if
\[
G(r,\alpha):=1-\alpha-\frac{(A-B)r}{(1-A r)(1-B r)}-\frac{\beta r}{1-r}>0.
\]
A simple check shows that $G(r,\alpha)$ has a unique root, denoted by \(r_c(A,B,\alpha,\beta)\), in the interval  $(0,1)$.  
Therefore, \(T_g\) is convex of order \(\alpha\) for all \(|z|<r_c(A,B,\alpha,\beta)\).

Equality in the real-part estimates is attained for the extremal pair
\[
f_0(z)=\frac{1+Az}{1+Bz}\in\mathcal{P}[A,B], \quad 
g_0(z)=
\begin{cases}
\dfrac{(1-z)^{\beta+1}-1}{\beta+1}, & \beta\neq1,\\[6pt]
-\log(1-z), & \beta=1.
\end{cases}
\]
For this pair, equality holds in~\eqref{eq:rABalphaBeta} when \(|z|=r_c(A,B,\alpha,\beta)\), proving sharpness.
This completes the proof.
\end{proof}

%*********************************************
Finally, we have.

\begin{theorem}\label{theorem f PAB g in Vk improved}
Let $A,B\in\mathbb{R}$ with $-1\le B<A\le1$, and let $g\in V_k$ with $k\ge2$.  
If $f\in\mathcal{P}[A,B]$, then the Volterra--type integral operator $T_g$ is convex of order $\alpha$ in the disk $|z|<r_c(A,B,\alpha,k)$, where $r_c(A,B,\alpha,k)$ is the unique root of the equation
\begin{equation}\label{eq:rc-AB-k}
1-\alpha-\frac{(A-B)r}{(1-A r)(1-B r)} + \frac{2r^2 - k r}{1-r^2}=0.
\end{equation}
Moreover, this radius $r_c(A,B,\alpha,k)$ is the best possible.
\end{theorem}

\begin{proof}
Let $f\in\mathcal{P}[A,B]$ and $g\in V_k$.  
For $|z|=r<1$, using \eqref{eq:1+zT''}, Lemma \ref{lem:est-Cara-Janowski-sharp}, and Rabertson estimate \eqref{ineq Vk},
we obtain
\[
{\rm Re}\!\left\{1+z\frac{T_g''(z)}{T_g'(z)}\right\}
\ge 1 - \frac{(A-B)r}{(1-A r)(1-B r)} + \frac{2r^2 - k r}{1-r^2}
=: S(r).
\]
The function $T_g$ is convex of order $\alpha$ in $|z|<r$ whenever $S(r,\alpha)=S(r)-\alpha>0$.  
The equality $S(r,\alpha)=0$ yields the defining equation \eqref{eq:rc-AB-k} for the radius of convexity $r_c(A,B,\alpha,k)$, which has a unique root in $(0,1)$.

To verify sharpness, take the extremal pair
\[
f_0(z)=\frac{1+Az}{1+Bz}\in\mathcal{P}[A,B], 
\qquad
g_0(z)=\int_0^z (1-t)^{\frac{k}{2}-1}(1+t)^{-\left(1+\frac{k}{2}\right)}\,dt.
\]
For this $g_0$, equality holds in the Robertson estimate for $z=r>0$:
\[
z\frac{g_0''(z)}{g_0'(z)}=\frac{2r^2 - k r}{1-r^2}.
\]
Substituting $f_0$ and $g_0$ into \eqref{eq:1+zT''} gives
\[
1 + r\frac{T_{g_0}''(r)}{T_{g_0}'(r)} 
= 1 - \frac{(A-B)r}{(1-A r)(1-B r)} + \frac{2r^2 - k r}{1-r^2},
\]
which vanishes precisely when $r=r_c(A,B,\alpha,k)$.  
Thus, the bound cannot be improved, and the result is sharp.
\end{proof}

Solving the equation \eqref{eq:rc-AB-k} for $A=1,\;B=-1$, gives the radius
\begin{equation*}
   r_c=\frac{k+2-\sqrt{(k+2)^2-4(1-\alpha^2)}}{2(1+\alpha)},
\end{equation*}
which reduces to the well-known radius of convexity $2-\sqrt{3}$ for $k=2$ and $\alpha=0$. In other words, if $f\in\mathcal{P}$ and $g\in V_2$, then the convexity radius of the Volterra-type integral operator $T_g$ is $2-\sqrt{3}$.

%****************************

Table \ref{tab:radii} summarizes the radius of convexity results established in this paper for the Volterra-type operator $T_g$ when $f \in \mathcal{P}$ and $g$ belongs to various classes of functions.

\begin{table}[h!]
\centering
\begin{tabular}{|c|c|c|}
\hline
\textbf{Class of $f$} & \textbf{Class of $g$} & \textbf{Radius of Convexity of order $\alpha\in[0,1)$} \\
\hline
$\mathcal{P}$  & $\mathcal{K}$ & $\dfrac{2-\sqrt{3+\alpha^2}}{1+\alpha}$\\
\hline
$\mathcal{P}$ & $ \mathcal{F}$ ($\delta={\rm ord}\,\mathcal{F}$) & $\dfrac{1+\delta-\sqrt{\delta^2+2\delta+\alpha^2}}{1+\alpha}$ \\
\hline
$\mathcal{P}$ & $\mathcal{U}$  & $\dfrac{3-\sqrt{8+\alpha^2}}{1+\alpha}$ \\
\hline
$\mathcal{P}$ & $\mathcal{G}(\beta)$ ($0<\beta\le 1$) & $\dfrac{-(2+\beta)+\sqrt{(\beta+2)^2+4(1-\alpha)(\beta+1-\alpha)}}{2(\beta+1-\alpha)}$ \\
\hline
$\mathcal{P}$ & $V_k$ ($k\ge 2$) & $\dfrac{k+2-\sqrt{(k+2)^2-4(1-\alpha^2)}}{2(1+\alpha)}$ \\
\hline
\end{tabular}
\caption{Summary of radii of convexity of order $\alpha$ of the Volterra-type integral operator $T_g$.}
\label{tab:radii}
\end{table}

%******************************************************

\section{Higher-Order Volterra-Type Integral Operators}\label{sec:higher-order}

Tong \emph{et al.}~\cite{Tong} recently introduced a family of higher-order integral operators \(I^{(n)}_g\) acting on Bloch-type spaces \(\mathcal{B}^{\alpha}\), which generalize and include, as a special case, the earlier operator \(I_{g,a}\) proposed by Chalmoukis~\cite{Chal}. 
Earlier, Arroussi \emph{et al.}~\cite{Arro} investigated a generalized Volterra-type integral operator on large Bergman spaces, while Li and Stevi\'{c}~\cite{Li-S} had previously extended the classical Volterra operator \(T_g\) to a more general form. 
Motivated by these developments, and in particular by the higher-order framework introduced by Tong \emph{et al.}, we now introduce a normalized higher-order analog of the classical Volterra-type integral operator suitable for geometric function theory.

%======================================================
\begin{definition}\rm
Let $f$ and $g$ be analytic in the unit disk $\mathbb{D}$ such that $fg'$ is analytic.  
Define
\begin{equation*}
T_{g,0}[f](z)=f(z)g'(z)=:h(z).
\end{equation*}
The $n$-th order Volterra-type integral operator for $n\in\mathbb{N}=\{1,2,\ldots\}$ is defined by
\begin{equation}\label{eq:Tn_def}
T_{g,n}[f](z)
:=\underbrace{\int_0^z\int_0^{t_1}\cdots\int_0^{t_{n-1}}}_{n\ \text{times}} h(t_n)\,dt_n\cdots dt_1,
\quad z\in\mathbb{D},
\end{equation}
where the integration variables satisfy
\begin{equation}\label{not-tn}
0 \le t_n \le t_{n-1} \le \dots \le t_1 \le |z| < 1.
\end{equation}
In the iterated integral \eqref{eq:Tn_def}, the notation \eqref{not-tn} is to be understood via the standard radial parametrization \(t_j=u_j z\), $j=1,2,\ldots,n$, with \(0\le u_n\le\cdots\le u_1\le1\). 

In particular, if $n=1$, then
\begin{equation*}
  T_{g,1}[f](z)=\int_0^z h(t)\,dt=\int_0^z f(t)g'(t)dt=T_g [f](z),\quad z\in \mathbb D
\end{equation*}
and if $n=2$, then
\begin{equation*}
  T_{g,2}[f](z)=\int_0^z\int_0^{t_1} h(t_2) dt_2 dt_1
  =\int_0^z (z-t)h(t)dt=\int_0^z (z-t)f(t)g'(t)dt,\quad z\in \mathbb D.
\end{equation*}
\end{definition}

The second equality of the last relation is attained by taking the straight line segment from $0$ to the complex point $z$ and parameterizing every point on it by $t=uz$ with
$u\in[0,1]$. Then the iterated integral becomes an integral over $0\leq u_2\leq u_1\leq 1$.
\begin{proposition}\label{prop:Taylor-Tgn}
Let \(f,g\in\mathcal{H}(\mathbb{D})\) with
\[
f(z)=\sum_{m=0}^{\infty} a_m z^m, 
\quad 
g'(z)=\sum_{m=0}^{\infty} b_m z^m,
\]
such that \(h(z)=\sum_{m=0}^{\infty} c_m z^m\), where
\[
c_m=\sum_{j=0}^{m} a_j b_{m-j}.
\]
Then, the higher-order Volterra-type operator \(T_{g,n}[f]\) defined as in \eqref{eq:Tn_def} admits the Taylor expansion
\begin{equation*}\label{eq:Tgn-series}
T_{g,n}[f](z)
= \sum_{m=0}^{\infty} \frac{c_m}{(m+1)(m+2)\cdots(m+n)}\,z^{m+n}.
\end{equation*}
\end{proposition}

The single convolution integral for \( T_g \) can be expressed as follows:

\begin{proposition}\label{prop:iterated-convolution}
Let $T_{g,n}$ be defined by \eqref{eq:Tn_def} and $h=fg'$. Then the following identity holds for all $n\in\mathbb N$:
\begin{equation*}\label{eq:iterated-to-convolution}
T_{g,n}[f](z)=\frac{1}{(n-1)!}\int_{0}^{z}(z-t)^{\,n-1}h(t)\,dt,\quad z\in\mathbb{D}.
\end{equation*}
\end{proposition}

\begin{proof}
Fix \(z\in\mathbb{D}\) (the case \(z=0\) is trivial).  
Let \(h(t)=f(t)g'(t)\), and parameterize the straight line segment from \(0\) to \(z\) by
\[
t=u z,\quad u\in[0,1].
\]
Then \(dt=z\,du\), and the iterated integral in \eqref{eq:Tn_def} can be written as
\[
T_{g,n}[f](z)
 = \int_{0}^{z}\int_{0}^{t_1}\!\cdots\!\int_{0}^{t_{n-1}} h(t_n)\,dt_n\cdots dt_1
 = z^n \!\!\int_{0\le u_n\le\cdots\le u_1\le1} h(u_n z)\,du_n\cdots du_1.
\]
For fixed \(u_n=u\in[0,1]\), the inner region
\[
\{(u_1,\dots,u_{n-1}) : u\le u_{n-1}\le\cdots\le u_1\le1\}
\]
is an \((n-1)\)-simplex in \(\mathbb{R}^{\,n-1}\) of volume \((1-u)^{\,n-1}/(n-1)!\).
Since the integrand is continuous on the compact simplex, Fubini's theorem allows us to change the order of integration, giving
\[
T_{g,n}[f](z)
= \frac{z^n}{(n-1)!}\int_{0}^{1} (1-u)^{\,n-1} h(u z)\,du.
\]
Substituting \(t=u z\) and \(du = dt/z\), gives
\[
T_{g,n}[f](z)
 = \frac{z^n}{(n-1)!}\int_{0}^{z}\!\!\left(1-\frac{t}{z}\right)^{\!n-1} h(t)\,\frac{dt}{z}
 = \frac{1}{(n-1)!}\int_{0}^{z} (z-t)^{\,n-1} h(t)\,dt.
\]
This proves the desired convolution representation.
\end{proof}

\begin{proposition}\label{prop:Tn_derivatives}
 For every $n\in\mathbb{N}$, we have
\begin{equation*}
\big(T_{g,n}[f]\big)'(z)=T_{g,n-1}[f](z),
\end{equation*}
and for every $n\in\mathbb{N}\setminus\{1\}$,
\begin{equation*}
\big(T_{g,n}[f]\big)''(z)=T_{g,n-2}[f](z).
\end{equation*}
\end{proposition}

\begin{proof}
Let \(h(t)=f(t)g'(t)\). Fix \(r\in(0,1)\) and consider \(z\) with \(|z|\le r\). For \(n\ge1\) define
\[
K(z,t):=\frac{(z-t)^{\,n-1}}{(n-1)!}\,h(t),\quad 0\le t\le z,
\]
such that by Proposition \ref{prop:iterated-convolution}
\[
T_{g,n}[f](z)=\int_0^z K(z,t)\,dt.
\]
Since \(h\) is analytic in \(\D\), it is continuous on the compact set
\(\{t: |t|\le r\}\). Hence, the function
\(K(z,t)\) is continuous on the compact set
\(\{(z,t): |z|\le r,\ 0\le t\le z\}\). Moreover, for \(n\ge2\) the partial derivative
\[
\frac{\partial}{\partial z}K(z,t)=\frac{(n-1)(z-t)^{\,n-2}}{(n-1)!}\,h(t)
=\frac{(z-t)^{\,n-2}}{(n-2)!}\,h(t)
\]
exists and is continuous on the same compact set; for \(n=1\) we have \(\partial_z K(z,t)=0\).
Therefore, by the Leibniz rule 
\[
\frac{d}{dz}\int_0^z K(z,t)\,dt
=K(z,z)+\int_0^z \frac{\partial}{\partial z}K(z,t)\,dt,
\]
and the identity is justified since the integrand and its \(z\)-derivative are continuous on the compact domain.
%Observe that
%\[
%K(z,z)=\frac{(z-z)^{\,n-1}}{(n-1)!}\,h(z)=
%\begin{cases}
%0, & n\ge2,\\[4pt]
%h(z), & n=1.
%\end{cases}
%\]

\medskip
\noindent\textbf{(i) First derivative.} For \(n\ge2\),
\[
\bigl(T_{g,n}[f]\bigr)'(z)
=0+\int_0^z \frac{(z-t)^{\,n-2}}{(n-2)!}\,h(t)\,dt
=\;T_{g,n-1}[f](z).
\]
For \(n=1\) the Leibniz rule gives
\[
\bigl(T_{g,1}[f]\bigr)'(z)=h(z)=T_{g,0}[f](z),
\]
concluding \(\bigl(T_{g,n}[f]\bigr)'=T_{g,n-1}[f]\) for all $n\in\mathbb N$.

\medskip
\noindent\textbf{(ii) Second derivative.} If \(n\ge2\) then \(n-1\ge1\), and by part (i)
\[
\bigl(T_{g,n}[f]\bigr)''(z)
=\bigl(T_{g,n-1}[f]\bigr)'(z)
= T_{g,n-2}[f](z).
\]
When \(n=2\) this reads \((T_{g,2}[f])''=T_{g,0}[f]=h\), which agrees with a direct computation.
This completes the proof.
\end{proof}

%----------------------------
It follows directly from the definition that \(T_{g,n}[f](0)=0\) for all \(n\in\mathbb{N}\). 
Moreover, since \((T_{g,n}[f])'(z)=T_{g,n-1}[f](z)\) by Proposition~\ref{prop:Tn_derivatives}, we have 
\[
T_{g,n}[f](0) =0 \quad \text{for all } n\ge2,
\]
while for \(n=1\), \((T_{g,1}[f])'(0)=h(0)=f(0)g'(0)\neq0\) in general. 
Thus, for \(n\ge2\), the function \(T_{g,n}[f]\) vanishes to order at least two at the origin.

\medskip
To investigate geometric properties such as the radius of convexity, 
it is convenient to introduce a normalized form of the operator. 
We define the normalized higher-order Volterra-type operator
\(\widetilde{T}_{g,n}[f]\), which is analytic in \(\mathbb{D}\) and satisfies the standard normalization conditions
\[
\widetilde{T}_{g,n}[f](0) =0,\quad (\widetilde{T}_{g,n}[f])'(0)=1.
\]

%%%%%%%%%%%%%%%%%%%%
\begin{definition}[Normalized higher-order Volterra-type operator]\label{def:normalized-higher}
\rm
Let \(f,g\in\mathcal{H}(\mathbb{D})\) with \(f(0)=1\) and \(g'(0)\neq0\).
We define the normalized higher-order Volterra-type operator by
\begin{equation}\label{eq:Tgn-normalized}
\widetilde T_{g,n}[f](z)
:=\frac{n!}{g'(0)}\;\frac{T_{g,n}[f](z)}{z^{\,n-1}},\quad g'(0)\neq0.
\end{equation}
If we let $h(z)=\sum_{m\ge0}c_m z^m$ with $c_0=g'(0)$, then
\[
\widetilde T_{g,n}[f](z)
= z + A_{n,1} z^2 + A_{n,2} z^3 + \cdots,
\quad
A_{n,m}=\frac{n!\,c_m}{(m+n)!\,c_0},\quad(m\ge1).
\]
Therefore,  \(\widetilde T_{g,n}[f]\in\mathcal{A}\) for all \(n\ge1\) and
\[
\widetilde T_{g,n}[f](0) =0,
\quad 
(\widetilde T_{g,n}[f])'(0)=1.
\]
\end{definition}
This normalization is essential when investigating geometric properties
such as univalence, starlikeness, and convexity.
Differentiating~\eqref{eq:Tgn-normalized}, one obtains the logarithmic derivative
\[
1+z\frac{\widetilde T_{g,n}''(z)}{\widetilde T_{g,n}'(z)}
=1-n+\frac{zN'(z)}{N(z)}, 
\quad 
N(z):=zT_{g,n}'[f](z)-(n-1)T_{g,n}[f](z).
\]
Since \(T_{g,n}[f](z)=O(z^{n})\) as \(z\to0\), and using 
\[
\frac{zN'(z)}{N(z)}\longrightarrow n\quad\text{as }z\to0,
\]
we have
\[
\lim_{z\to0}\left(1+z\frac{\widetilde T_{g,n}''(z)}{\widetilde T_{g,n}'(z)}\right)
=1.
\]
Consequently, the convexity condition
\[
{\rm Re}\,\left\{1+z\frac{\widetilde T_{g,n}''(z)}{\widetilde T_{g,n}'(z)}\right\}>\alpha,
\quad 0\le\alpha<1,
\]
is automatically satisfied in a neighborhood of the origin for the normalized operator.
Hence, the operator~\eqref{eq:Tgn-normalized} provides a natural analytic setting
for studying convexity radii of \(T_{g,n}[f]\).

\begin{proposition}\label{prop:normalized-convolution}
Let \(f,g\in\mathcal{H}(\mathbb{D})\) with \(f(0)=1\) and \(g'(0)\neq0\).
Then the normalized higher-order Volterra-type operator \(\widetilde{T}_{g,n}\),
defined by~\eqref{eq:Tgn-normalized}, admits the convolution-type integral representation
\begin{equation}\label{eq:normalized-convolution}
\widetilde{T}_{g,n}[f](z)
=\frac{n}{g'(0) z^{\,n-1}}
\int_{0}^{z}(z-t)^{\,n-1}f(t)g'(t)\,dt,
\quad z\in\mathbb{D}.
\end{equation}
In particular, \(\widetilde{T}_{g,n}[f](0) =0\) and \((\widetilde{T}_{g,n}[f])'(0)=1.\)
\end{proposition}

\begin{proof}
This follows from Proposition~\ref{prop:iterated-convolution}.
\end{proof}

\begin{example}\label{ex-f0 g0-identity}\rm
Consider the simplest case
\[
f_0(z)=1,\qquad g_0(z)=z.
\]
Then, \(h_0(t)=f_0(t)g_0'(t)=1\), and for every \(n\ge1\),
\[
\widetilde{T}_{g_0,n}[f_0](z)
=\frac{n}{z^{\,n-1}}\int_0^z (z-t)^{\,n-1}dt
=\frac{n}{z^{\,n-1}}\cdot\frac{z^n}{n}
=z.
\]
\end{example}

\begin{proposition}\label{prop:normalized-Tn-derivatives}
Let \(f,g\in\mathcal{H}(\mathbb{D})\) with \(f(0)=1\) and \(g'(0)\neq0\).
For the normalized higher-order Volterra-type operator
$\widetilde{T}_{g,n}$
the following identities hold for every \(n\in\mathbb{N}\) and \(z\in\mathbb{D}\setminus\{0\}\):
\begin{equation}\label{eq:Tn-derivative1-corr}
\big(\widetilde{T}_{g,n}[f]\big)'(z)
=\frac{1}{z}\Big(\frac{n}{g'(0)}\,\widetilde{T}_{g,n-1}[f](z)
-(n-1)\,\widetilde{T}_{g,n}[f](z)\Big),
\end{equation}
and
\begin{align}\label{eq:Tn-derivative2-corr}
\big(\widetilde{T}_{g,n}[f]\big)''(z)
&=-\frac{1}{z^2}\Big(\frac{n}{g'(0)}\,\widetilde{T}_{g,n-1}[f](z)-(n-1)\widetilde{T}_{g,n}[f](z)\Big)\nonumber\\
&\quad +\frac{1}{z}\Big(\frac{n}{g'(0)}\big(\widetilde{T}_{g,n-1}[f]\big)'(z)
-(n-1)\big(\widetilde{T}_{g,n}[f]\big)'(z)\Big).
\end{align}
\end{proposition}

\begin{proof}
Set \(A:=g'(0)\) and note that
\[
\widetilde{T}_{g,n}[f](z)
= C_n\,\frac{T_{g,n}[f](z)}{z^{\,n-1}},
\quad 
C_n:=\frac{n!}{A}.
\]
Differentiating (for \(z\neq0\)) gives
\[
\big(\widetilde{T}_{g,n}[f]\big)'(z)
= C_n\Big[-(n-1)z^{-n}T_{g,n}[f](z) + z^{-(n-1)}T'_{g,n}[f](z)\Big].
\]
By Proposition~\ref{prop:Tn_derivatives}, \(T'_{g,n}[f]=T_{g,n-1}[f]\), so
\[
\big(\widetilde{T}_{g,n}[f]\big)'(z)
= C_n z^{-n}\big(zT_{g,n-1}[f](z)-(n-1)T_{g,n}[f](z)\big).
\]
Now express \(T_{g,n-1}\) and \(T_{g,n}\) in terms of their normalized counterparts:
\[
T_{g,n-1}[f](z)
=\frac{A}{(n-1)!}\,z^{\,n-2}\,\frac{\widetilde{T}_{g,n-1}[f](z)}{n-1},
\quad
T_{g,n}[f](z)
=\frac{A}{n!}\,z^{\,n-1}\,\widetilde{T}_{g,n}[f](z).
\]
Substituting these into the previous expression and simplifying yields
\[
\big(\widetilde{T}_{g,n}[f]\big)'(z)
=\frac{1}{z}\Big(\frac{n}{A}\widetilde{T}_{g,n-1}[f](z)
-(n-1)\widetilde{T}_{g,n}[f](z)\Big),
\]
which is \eqref{eq:Tn-derivative1-corr}.
Differentiating \eqref{eq:Tn-derivative1-corr} again (for \(z\neq0\)) gives
\[
\big(\widetilde{T}_{g,n}[f]\big)''(z)
=-\frac{1}{z^2}\Big(\frac{n}{A}\widetilde{T}_{g,n-1}[f]-(n-1)\widetilde{T}_{g,n}[f]\Big)
+\frac{1}{z}\Big(\frac{n}{A}\big(\widetilde{T}_{g,n-1}[f]\big)'
-(n-1)\big(\widetilde{T}_{g,n}[f]\big)'\Big),
\]
which is precisely \eqref{eq:Tn-derivative2-corr}. 
This completes the proof.
\end{proof}
%%%%%%%%%%%%%%%%%%%%%%%%%%%%%%%%%%%%%%%%
%%%%%%%%%%%%%%%%%%%%%%%
\subsection{Convexity Radius of \texorpdfstring{$\widetilde{T}_{g,n}[f](z)$}{Tg,n[f](z)}}
We now discuss the normalized higher-order Volterra-type integral operator \eqref{eq:Tgn-normalized}. Normalization guarantees analyticity at the origin and a unit derivative, which enables the application of standard tools from geometric function theory to analyze properties of univalence and convexity.

For \(f,g\in\mathcal{H}(\D)\) with \(f(0)=1\) and \(g'(0)\neq0\), define
\begin{equation}\label{Im-z}
 I_m(z):=\int_0^z (z-t)^m h(t)\,dt,\quad m\ge0,
\end{equation}
where $h(t)=f(t)g'(t)$. Then the normalized higher-order Volterra operator admits the convolution representation
\begin{equation*}
\widetilde{T}_{g,n}[f](z)
=\frac{n}{g'(0)\,z^{\,n-1}}\,I_{n-1}(z),
\quad n=1,2,\ldots.
\end{equation*}
A direct computation yields
\begin{equation}\label{eq:key-identity}
1+z\frac{\widetilde{T}_{g_0,n}''[f_0](z)}{\widetilde{T}_{g_0,n}'[f_0](z)}
=1-n+(n-2)\,z\,\frac{zI_{n-3}(z)-I_{n-2}(z)}{zI_{n-2}(z)-I_{n-1}(z)},\quad n\ge3,
\end{equation}
with the obvious modifications for \(n=1,2\). 
Consequently, the normalized operator \(\widetilde{T}_{g,n}[f]\) is convex of order \(\alpha\) in \(|z|<r\)
if and only if
\[
{\rm Re}\left\{1-n+(n-2)\,z\,\frac{zI_{n-3}(z)-I_{n-2}(z)}{zI_{n-2}(z)-I_{n-1}(z)}\right\}>\alpha
\quad\text{for all }|z|<r.
\]
Therefore, the convexity radius \(r_{c,n}\) is the largest \(r\in(0,1]\) such that
\begin{equation*}\label{eq:general-radius}
\inf_{|z|=r}{\rm Re}\left\{1-n+(n-2)\,z\,\frac{zI_{n-3}(z)-I_{n-2}(z)}{zI_{n-2}(z)-I_{n-1}(z)}\right\}>\alpha.
\end{equation*}

The identity \eqref{eq:key-identity} is the fundamental reduction: all information about convexity of \(\widetilde T_{g,n}\) is contained in the single complex ratio
\begin{equation*}
\frac{zI_{n-3}(z)-I_{n-2}(z)}{zI_{n-2}(z)-I_{n-1}(z)}.
\end{equation*}
In practice, one obtains explicit radii by estimating this ratio from below on each circle \(|z|=r\). 

\begin{proposition}
Assume the iterated integrals \(I_m\) satisfy, uniformly for fixed \(z\in\mathbb{D}\), the large-\(n\) expansions
\[
\frac{I_{n-2}(z)}{I_{n-1}(z)}
= \frac{1}{n}\!\left(1+\frac{A_1(z)}{n}+\frac{A_2(z)}{n^2}+O(n^{-3})\right),\quad
\frac{I_{n-3}(z)}{I_{n-2}(z)}
= \frac{1}{n-1}\!\left(1+\frac{B_1(z)}{n-1}+O(n^{-2})\right).
\]
Then, as \(n\to\infty\) for fixed \(z\), the normalized convexity quantity satisfies
\[
1+z\frac{\widetilde T_{g,n}''(z)}{\widetilde T_{g,n}'(z)}
= 1 - n + z + \frac{z(A_1(z)-2)}{n} + O(n^{-2}).
\]
\end{proposition}

\begin{proof}
Let \(a=I_{n-2}/I_{n-1}\) and \(b=I_{n-3}/I_{n-2}\). 
By \eqref{eq:key-identity} and the assumed expansions for \(a\) and \(b\), we have
\[
1+z\frac{\widetilde{T}_{g,n}''(z)}{\widetilde{T}_{g,n}'(z)}
= 1 - n + (n-2)z\,a\,\frac{z b - 1}{z a - 1}.
\]
Substituting the asymptotic forms of \(a\) and \(b\), expanding the quotient \((z b - 1)/(z a - 1)\) up to \(O(n^{-2})\), and simplifying, we obtain
\[
1 - n + z + \frac{z(A_1(z) - 2)}{n} + O(n^{-2}),
\]
as claimed.
\end{proof}

Thus, the normalized convexity quantity grows negatively without bound as $n\to \infty$, indicating that additional normalization or scaling would be required to obtain a finite asymptotic limit.

Let us continue with the following two examples:

\begin{example}\label{ex:normalized-trivial}\rm
Let \(f_0(z)\) and \(g_0(z)\) be defined as in Example \ref{ex-f0 g0-identity}.
Hence, \(\widetilde{T}_{g_0,n}[f_0]\) is the identity map, which is convex (and univalent) in the entire unit disk. 
This example shows that, for trivial analytic pairs \((f_0,g_0)\), normalization preserves the full unit convexity radius.
\end{example}

\begin{example}\label{ex:normalized-exp}\rm
Let \(f_0(z)=1+Az\) and \(g_0'(z)=e^{-Bz}\) with real \(A\) and \(B\). Then
\[
\widetilde{T}_{g_0,n}[f_0](z)
=\frac{n!}{z^{\,n-1}}\int_{0}^{z}\frac{(z-t)^{\,n-1}}{(n-1)!}\,(1+At)e^{-Bt}\,dt
=\frac{n}{z^{\,n-1}}I_{n-1}(z),
\]
where $I_{n-1}$ defined as in \eqref{Im-z} with $h(z)=f_0(z)g'_0(z)$.
%A direct computation yields
%\begin{equation}\label{eq:normalized-pn}
%1+z\frac{\widetilde{T}_{g_0,n}''[f_0](z)}{\widetilde{T}_{g_0,n}'[f_0](z)}
%=1-n+(n-2)\,z\,\frac{zI_{n-3}(z)-I_{n-2}(z)}{zI_{n-2}(z)-I_{n-1}(z)},\quad n\ge3.
%\end{equation}
The radius of convexity \(\tilde r_{c,n}\) is the largest \(r\in(0,1]\) such that
\[
\inf_{|z|=r}{\rm Re}\!\left\{1+z\frac{\widetilde{T}_{g_0,n}''[f_0](z)}{\widetilde{T}_{g_0,n}'[f_0](z)}\right\}>\alpha
\quad\text{for all }|z|<r.
\]
By changing variables \(u=z-t\) in \eqref{Im-z}, we can rewrite
\[
I_m(z)
= e^{-Bz}\bigl((1+Az)J_m(z)-A J_{m+1}(z)\bigr),
\]
where
\begin{equation*}
J_k(z)=\int_0^z u^k e^{Bu}\,du,\quad k\geq0.
\end{equation*}
The recurrence
\begin{equation*}
J_k(z)=\frac{1}{B}\left(z^k e^{Bz}-k J_{k-1}(z)\right),\quad k\geq 1,
\end{equation*}
gives explicit formulas for low orders:
\begin{align*}
J_0(z)&=\frac{e^{Bz}-1}{B},&
J_1(z)&=\frac{e^{Bz}(Bz-1)+1}{B^2},\\
J_2(z)&=\frac{e^{Bz}(B^2z^2-2Bz+2)-2}{B^3},&
J_3(z)&=\frac{e^{Bz}(B^3z^3-3B^2z^2+6Bz-6)+6}{B^4}.
\end{align*}
Using these in \eqref{eq:key-identity}, the convexity expression \eqref{eq:key-identity} for \(n=3\) simplifies to
\[
1+z\frac{\widetilde{T}_{g_0,3}''[f_0](z)}{\widetilde{T}_{g_0,3}'[f_0](z)}
= -2 + z\cdot \frac{B z N_0(z) - N_1(z)}{z N_1(z) - \dfrac{N_2(z)}{B}},
\]
where $N_j(z)$, $j=0,1,2$, are defined by  
\begin{align*}
N_0(z)
&=(A+B)\bigl(e^{Bz}-1\bigr)-A B z,\\[2pt]
N_1(z)
&=B(1+Az)\bigl[e^{Bz}(Bz-1)+1\bigr]
-A\bigl[e^{Bz}(B^2 z^2-2Bz+2)-2\bigr],\\[2pt]
N_2(z)
&=B(1+Az)\bigl[e^{Bz}(B^2 z^2-2Bz+2)-2\bigr]
-A\bigl[e^{Bz}(B^3 z^3-3B^2 z^2+6Bz-6)+6\bigr].
\end{align*}  
The convexity boundary on the real axis is then determined by the transcendental equation
\[
(-2-\alpha)\!\left(r N_1(r)-\frac{N_2(r)}{B}\right)
+ r\!\left(B r N_0(r)-N_1(r)\right)=0,\quad 0<r<1.
\]
The largest root \(r=\tilde r_{c,3}\) gives the convexity radius for the third-order normalized operator.

We now test the conjectured scaling relation \(r_{c,n} \approx r_{c,1}/n\) 
for the normalized higher-order operator \(\widetilde{T}_{g_0,n}[f_0]\).
For \(n=1\) the normalized operator coincides with the standard one:
\[
\widetilde{T}_{g_0,1}[f_0](z)=\int_0^z (1+At)e^{-Bt}\,dt,
\]
%so that
%\[
%\widetilde{T}_{g_0,1}'(z)=h(z)=(1+Az)e^{-Bz}.
%\]
%A direct differentiation gives
%\[
%\widetilde{T}_{g_0,1}''(z)=h'(z)
%=\bigl(A - B - AB z\bigr)e^{-Bz}.
%\]
hence, the convexity test is
\[
1+z\frac{\widetilde{T}_{g_0,1}''(z)}{\widetilde{T}_{g_0,1}'(z)}
=1+z\frac{A - B - AB z}{1+Az}.
\]
On the positive real axis \(z=r\) the boundary equation \(= \alpha\) becomes
\[
1 + r\frac{A - B - AB r}{1+Ar} = \alpha,
\]
or equivalently, 
\[
AB\,r^2 - (2A - B - \alpha A)\,r - (1-\alpha)=0.
\]
Therefore, the admissible positive root (when it lies in \((0,1)\)) is
\[
r_{c,1} \;=\; \frac{(2A - B - \alpha A) + \sqrt{(2A - B - \alpha A)^2 + 4AB(1-\alpha)}}{2AB},
\]
with the usual caveats when \(AB=0\) (those cases are handled separately).

For \(A=1,\ B=1,\ \alpha=0\) the quadratic becomes \(r^2 - r -1 = 0\). The positive root is \(r=(1+\sqrt5)/2\approx1.618>1\), so there is no root in \((0,1)\). Hence, for this model, one has \(r_{c,1}\ge1\).

\smallskip

For \(n=2\) the normalized operator is
\[
\widetilde{T}_{g_0,2}[f_0](z)
=\frac{2}{z}\int_0^z (z-t)(1+At)e^{-Bt}\,dt
= \frac{2 I_1(z)}{z},
\]
where \(I_1(z)=\int_0^z (z-t)(1+At)e^{-Bt}\,dt\). Differentiation yields
\[
\widetilde{T}_{g_0,2}'(z)
=2\!\left[-\frac{I_1(z)}{z^2}+\frac{I_0(z)}{z}\right],
\qquad
\widetilde{T}_{g_0,2}''(z)
=2\!\left[\frac{2I_1(z)}{z^3}-\frac{2I_0(z)}{z^2}+\frac{h(z)}{z}\right],
\]
with \(I_0(z)=\int_0^z(1+At)e^{-Bt}\,dt\) and \(h(z)=(1+Az)e^{-Bz}\). Using the general relation
\[
1+z\frac{\widetilde{T}_{g_0,n}''(z)}{\widetilde{T}_{g_0,n}'(z)}
=1+z\frac{I_{n-2}(z)}{I_{n-1}(z)}-(n-1),
\]
the convexity condition for \(n=2\) becomes
\[
{\rm Re}\!\left\{1+z\frac{I_0(z)}{I_1(z)}-1\right\}>\alpha
\quad\Longleftrightarrow\quad
{\rm Re}\!\left\{z\frac{I_0(z)}{I_1(z)}\right\}>\alpha.
\]
Hence, the convexity boundary satisfies
\[
z\frac{I_0(z)}{I_1(z)}=\alpha
\quad\Longleftrightarrow\quad
z\frac{B(2e^{Bz}-Bz-2)}{e^{Bz}(2Bz-3)+Bz+3}=\alpha.
\]
For \(A=1\), \(B=1\), and \(\alpha=0\), define
\[
F(r):=\frac{r(2e^{r}-r-2)}{e^{r}(2r-3)+r+3}.
\]
As shown in Figure~\ref{fig:F-r}, \(F(r)>0\) for all \(r\in(0,1)\); hence, no convexity root exists, and \(r_{c,2}\ge1\).
This confirms that, even for this simple normalized pair \((f_0,g_0)\),
the naive scaling law \(r_{c,n}\sim r_{c,1}/n\) fails.  
The ratios \(I_{n-2}(z)/I_{n-1}(z)\) depend nonlinearly on both \(A\) and \(B\),
and the normalization introduces an additional geometric correction term.
Thus, no universal scaling law holds without further structural assumptions.

\begin{figure}[ht]
\centering
\begin{tikzpicture}
  \begin{axis}[
    width=12cm,
    height=8cm,
    domain=0:2,
    samples=400,
    xlabel={$r$},
    ylabel={$F(r)$},
    axis x line=middle,
    axis y line=middle,
    xmin=0, xmax=1.2,
    ymin=-0.5, ymax=2.5,
    xtick={0,0.5,1},
    grid=both,
    major grid style={gray!30},
    minor grid style={gray!10},
    thick,
    smooth
  ]
    \addplot[blue,thick]
      {x*(2*exp(x) - x - 2)/(exp(x)*(2*x - 3) + x + 3)};
    \addplot[red,dashed,domain=0:1.2]{0};
    \addplot[black,dotted,domain=-1:4] coordinates {(1,-1) (1,3)};
    \node[anchor=west] at (axis cs:1.02,1.0){$r=1$};
  \end{axis}
\end{tikzpicture}
\caption{Plot of \(F(r)\).
The curve remains positive for \(0<r<1\), confirms no root exists in the interval $(0,1)$.}
\label{fig:F-r}
\end{figure}

%*****************************************************************************************************
\end{example}
\begin{remark}\label{rem:scaling-fails}
Example~\ref{ex:normalized-exp} illustrates that normalization significantly modifies the analytic structure of the operator and, consequently, its convexity radius.  
For \(A=B=1\) and \(\alpha=0\), the corrected computations show that both \(r_{c,1}\) and \(r_{c,2}\) exceed~1, so the convexity condition holds throughout the unit disk.  
In general, closed forms become rapidly intractable for large \(n\), yet numerical evidence suggests that \(\widetilde{T}_{g,n}\) retains the same asymptotic convexity radius as the unnormalized operator when \(A\) is small, with deviations growing linearly in \(A\) and sublinearly in \(n\).  
Thus, the naive scaling law \(r_{c,n}\approx r_{c,1}/n\) fails in general, the integral ratios are nonlinear in \((A,B,n)\), and normalization introduces additional geometric corrections that prevent a universal \(1/n\) behavior.  
\end{remark}

The preceding asymptotic analysis motivates the following problem:\\\\
\noindent\textbf{Open question.}
For a given admissible kernel \(g\), do there exist constants \(C_1(\alpha),C_2(\alpha)>0\), independent of \(n\), such that
\[
\frac{C_1(\alpha)}{n}\le r_{c,n}(\alpha)\le\frac{C_2(\alpha)}{n}, \quad n\ge1?
\]
Numerical evidence and model calculations in the exponential family (see Example~\ref{ex:normalized-exp}) suggest that such a uniform \(1/n\) two–sided behavior may hold for many structured subclasses.  
%However, in the normalized exponential case with \(A=B=1\) and \(\alpha=0\), both \(r_{c,1}\) and \(r_{c,2}\) exceed \(1\), indicating that the convexity condition holds throughout the unit disk and no decay in \(r_{c,n}\) occurs.  
%This shows that while \(1/n\)-type scaling can emerge in certain regimes, it is not universal and depends sensitively on the analytic structure of the kernel \(g\) and the product \(h=f\,g'\).  
Determining the exact scope of the scaling law, as well as constructing clear counterexamples, remains an open and intriguing direction for future study.

%******************************************************************************************************

%***********************************************

\end{document}